\documentclass[11pt,a4paper]{article}
\usepackage{amsmath,amssymb,amsthm}
\usepackage{graphicx,psfrag,rotating}
\usepackage{caption,subcaption}
\usepackage[noend]{algpseudocode}
\usepackage{enumerate}
\usepackage{hyperref}

\newcommand{\Z}{{{\mathbb{Z}}}}
\newcommand{\R}{{{\mathbb{R}}}}

\newcommand{\cL}{{{\mathcal{L}}}}
\newcommand{\cF}{{{\mathcal{F}}}}
\newcommand{\cC}{{{\mathcal{C}}}}

\newcommand{\MCG}{{\operatorname{MCG}}}

\newcommand{\ind}[1]{{^{{(#1)}}}}
\newcommand{\indp}[1]{\ind{#1}\vphantom{\cL}'}

\newtheorem{thm}{Theorem}
\newtheorem{lem}[thm]{Lemma}

\theoremstyle{definition}		
\newtheorem{defn}[thm]{Definition}
\newtheorem{alg}[thm]{Algorithm}
\newtheorem{notn}[thm]{Notation}

\theoremstyle{remark}
\newtheorem{remark}[thm]{Remark}

\errorcontextlines\maxdimen

\makeatletter
    \newcommand*{\algrule}[1][\algorithmicindent]{\makebox[#1][l]{\hspace*{.5em}\thealgruleextra\vrule height \thealgruleheight depth \thealgruledepth}}%
\newcommand*{\thealgruleextra}{}
\newcommand*{\thealgruleheight}{.75\baselineskip}
\newcommand*{\thealgruledepth}{.25\baselineskip}

\newcount\ALG@printindent@tempcnta
\def\ALG@printindent{%
    \ifnum \theALG@nested>0
        \ifx\ALG@text\ALG@x@notext
        \else
            \unskip
            \addvspace{-5.5 pt}
            \ALG@printindent@tempcnta=1
            \loop
                \algrule[\csname ALG@ind@\the\ALG@printindent@tempcnta\endcsname]%
                \advance \ALG@printindent@tempcnta 1
            \ifnum \ALG@printindent@tempcnta<\numexpr\theALG@nested+1\relax
            \repeat
        \fi
    \fi
    }%
\usepackage{etoolbox}
\patchcmd{\ALG@doentity}{\noindent\hskip\ALG@tlm}{\ALG@printindent}{}{\errmessage{failed to patch}}
\makeatother

\newbox\statebox
\newcommand{\myState}[1]{%
    \setbox\statebox=\vbox{#1}%
    \edef\thealgruleheight{\dimexpr \the\ht\statebox+1pt\relax}%
    \edef\thealgruledepth{\dimexpr \the\dp\statebox+1pt\relax}%
    \ifdim\thealgruleheight<.75\baselineskip
        \def\thealgruleheight{\dimexpr .75\baselineskip+1pt\relax}%
    \fi
    \ifdim\thealgruledepth<.25\baselineskip
        \def\thealgruledepth{\dimexpr .25\baselineskip+1pt\relax}%
    \fi
    \State #1%
    \def\thealgruleheight{\dimexpr .75\baselineskip+1pt\relax}%
    \def\thealgruledepth{\dimexpr .25\baselineskip+1pt\relax}%
}

\title{Intersections of multicurves from Dynnikov coordinates}
\author{S. \"Oyk\"u Yurtta\c s\footnote{Dicle University, Science Faculty,
    Mathematics Department, 21280, Diyarbak\i r, Turkey}\ ~and Toby
    Hall\footnote{Department of Mathematical Sciences, University of Liverpool,
    Liverpool L69 7ZL, UK}}
\date{November 2017}

\begin{document}
\maketitle

\begin{abstract} We present an algorithm for calculating the geometric
  intersection number of two multicurves on the $n$-punctured disk, taking as
  input their Dynnikov coordinates. The algorithm has complexity $O(m^2n^4)$,
  where~$m$ is the sum of the absolute values of the Dynnikov coordinates of
  the two multicurves. The main ingredient is an algorithm due to
  Cumplido for relaxing a multicurve.
\end{abstract}

\section{Introduction} 

Determining the geometric intersection number of two simple closed curves, or
of two multicurves (also known as \emph{integral laminations}), on a
surface~$S$ is a fundamental problem in computational topology. Algorithms such
as those of Bell and Webb~\cite{Bell-Webb} and Schaefer, Sedgwick, and
\v{S}tefankovi\v{c}~\cite{SSS} take as input the \emph{normal coordinates} of
the multicurves: vectors of minimal intersection numbers with the edges of an
ideal triangulation of~$S$. They compute the geometric intersection number of
two multicurves with complexity polynomial in the Euler characteristic of~$S$
and in $\log M$, where~$M$ is the sum of the normal coordinates. 

In this paper we restrict to the case where~$S=D_n$ is an $n$-punctured disk.
In this setting, multicurves are beautifully described by their \emph{Dynnikov
coordinates}~\cite{D02}: a collection of $2n-4$ linear combinations of
intersection numbers with the $3n-5$ edges of a near-triangulation, which
provide a bijection between the set of multicurves on~$D_n$ and $\Z^{2n-4}$. We
describe an algorithm for calculating the geometric intersection number of two
multicurves on~$D_n$ whose complexity is polynomial in~$n$ and in~$m$, the sum
of the absolute values of the coordinates. The advantages of this algorithm are
that it works directly with Dynnikov coordinates, and that it is
straightforward to express and to code.

The main ingredient is an algorithm due to Cumplido~\cite{Cumplido} which {\em
relaxes} a multicurve~$\cL$: that is, it finds a mapping class on $D_n$
(expressed as a positive braid) which sends~$\cL$ to a multicurve each of whose
components only intersects the horizontal diameter of the disk twice. Since the
geometric intersection number is invariant under the action of the mapping
class group, it only remains to provide an algorithm to calculate the geometric
intersection number of an arbitrary multicurve with a relaxed one.

Section~\ref{sec:ildc} is a brief introduction to multicurves, Dynnikov
coordinates, and the \emph{update rules} which describe the action of the braid
group~$B_n$ on Dynnikov coordinates. In Section~\ref{sec:int-num-elementary},
we derive a formula for calculating the geometric intersection number with a
relaxed multicurve (this is a corrected version of a formula described in
Theorem~11 of~\cite{paper2}). Cumplido's algorithm is stated in
Section~\ref{sec:cumplido}, and in Section~\ref{sec:the-alg} we state our
algorithm and analyse its complexity.

We work throughout with \emph{extended} Dynnikov coordinates in $\Z^{2n}$,
obtained by adjoining~4 redundant coordinates to the standard Dynnikov ones,
which brings computational and notational advantages. For the sake of
brevity, we refer to these extended coordinates simply as Dynnikov coordinates,
and to the usual Dynnikov coordinates as \emph{reduced}.
Remark~\ref{rmk:reduced-coords} provides formul\ae\ for translating between the
two types of coordinates.

\section{Multicurves and Dynnikov coordinates}
\label{sec:ildc}
\subsection{Multicurves on the punctured disk} Let $n\ge 3$, and $D_n$
be a standard model of the $n$-punctured disk in the plane, with the punctures
arranged along the horizontal diameter (henceforth referred to simply as
\emph{the diameter}). A simple closed curve in $D_n$ is {\em inessential} if it
bounds an unpunctured disk, a once-punctured disk, or an $n$-punctured disk,
and is {\em essential} otherwise.

A {\em multicurve} $\cL$ in $D_n$ is a finite union of pairwise
disjoint unoriented essential simple closed curves in $D_n$, up to isotopy
(that is, $\cL$ is the isotopy class of such a union of simple closed curves).
We write $\cL_n$ for the set of multicurves on $D_n$ (including the empty multicurve).

Given two multicurves~$\cL\ind{1},\cL\ind{2}\in\cL_n$ we write
\[
  \iota(\cL\ind{1}, \cL\ind{2}) =
    \min\{\# L\ind{1}\cap L\ind{2}\,:\, L\ind{1}\in\cL\ind{1} \text{ and }
    L\ind{2}\in\cL\ind{2}\},
\] the \emph{geometric intersection number} of $\cL\ind{1}$ and $\cL\ind{2}$.
The aim of this paper is to describe an algorithm for calculating
$\iota(\cL\ind{1}, \cL\ind{2})$ from the {\em Dynnikov coordinates} of
$\cL\ind{1}$ and~$\cL\ind{2}$.

We will regard~$n$ as being fixed throughout, and suppress the dependence of
some objects upon it.

\subsection{The Dynnikov coordinate system} The Dynnikov coordinate
system~\cite{D02} provides, for each $n\ge 3$, a bijection \mbox{$\rho_r\colon
\cL_n\to \Z^{2n-4}$}, which we now define (see Remark~\ref{rmk:reduced-coords}
below).

Construct {\em Dynnikov arcs} $\alpha_i$ ($-1\le i\le 2n-2$) and $\beta_i$
($0\le i\le n$) in $D_n$ as depicted in Figure~\ref{fig:dynn-arcs}. (The
unconventional indexing starting with $i=-1$ is to maintain consistency with
reduced Dynnikov coordinates, where the arcs $\alpha_{-1}$, $\alpha_0$,
$\alpha_{2n-3}$, $\alpha_{2n-2}$, $\beta_0$, and $\beta_n$ are not used.) Given
$\cL\in\cL_n$, let $L$ be a representative of $\cL$ which intersects each of
these arcs minimally (such an~$L$ is called a {\em minimal representative} of
$\cL$). Write $\alpha_i$ (respectively $\beta_i$) for the number of
intersections of $L$ with the arc $\alpha_i$ (respectively the arc $\beta_i$).
This overload of notation will not give rise to any ambiguity, since it will
always be stated explicitly when the symbols $\alpha_i$ and $\beta_i$ refer to
arcs rather than to integers. We write $(\alpha\,;\,\beta) =
(\alpha_{-1},\ldots,\alpha_{2n-2}\,;\,\beta_0,\ldots,\beta_n)$ for the
collection of intersection numbers associated to~$\cL$.

\begin{figure}[htbp]
\begin{center}
  \includegraphics[width=0.9\textwidth]{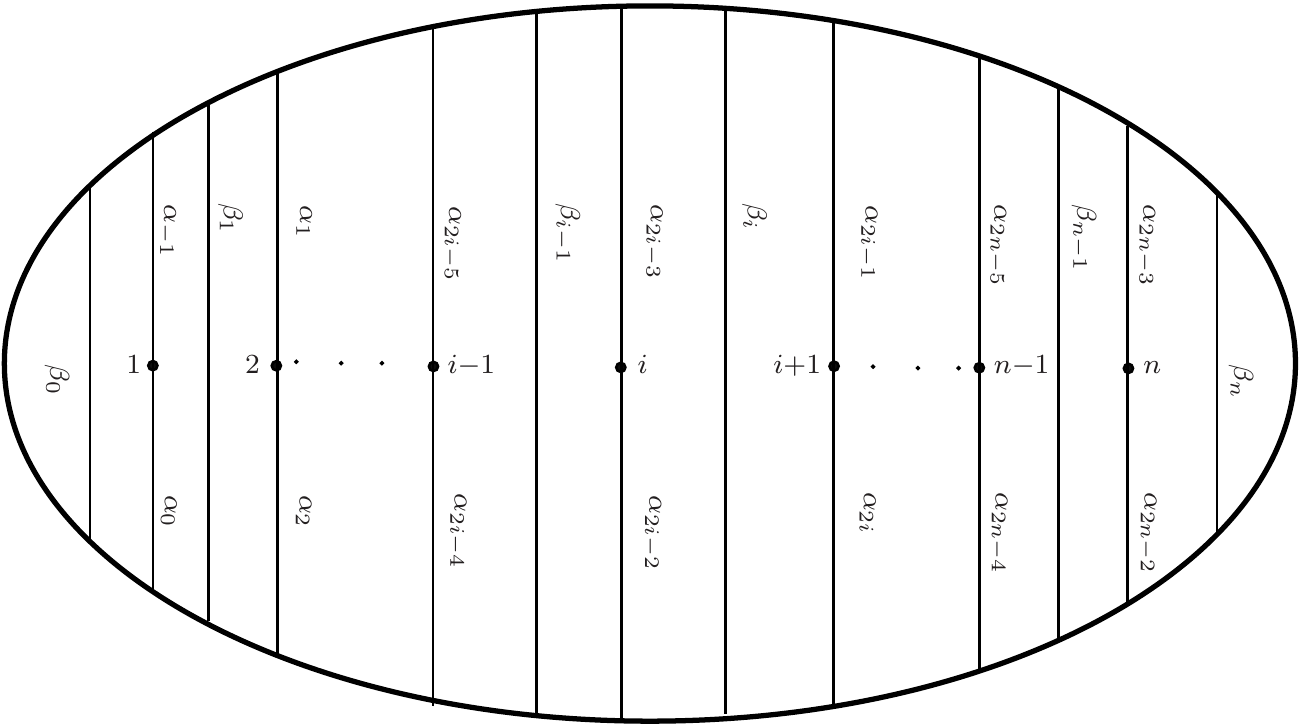}
  \caption{The arcs $\alpha_i$ and $\beta_i$}
  \label{fig:dynn-arcs}
\end{center}
\end{figure}

The {\em (extended) Dynnikov coordinate function} $\rho\colon\cL_n\to\Z^{2n}$
is defined by
\[
\rho(\cL) = (a;\,b) = (a_0,\ldots,a_{n-1};\,b_0,\ldots,b_{n-1}),
\] where
\begin{equation}
\label{eq:dynn-coords} a_i = \frac{\alpha_{2i}-\alpha_{2i-1}}{2}
\qquad\text{and}\qquad b_i = \frac{\beta_i - \beta_{i+1}}{2} \qquad
\end{equation} for $0\le i\le n-1$.

\medskip

Given $1\le i\le n$, let $\Delta_i$ denote the subset of~$D_n$ bounded by the
arcs $\beta_{i-1}$ and $\beta_i$ (which contains puncture~$i$). Let $L$ be a
minimal representative of~$\cL$, and consider the connected components of
$L\cap\Delta_i$. By minimality, each such component is of one of four types:
\begin{itemize}
  \item A {\em right loop} component, which has both endpoints on the arc
    $\beta_{i-1}$ and intersects both of the arcs $\alpha_{2i-3}$ and
    $\alpha_{2i-2}$;

  \item A {\em left loop} component, which has both endpoints on the arc
    $\beta_i$ and intersects both of the arcs $\alpha_{2i-3}$ and
    $\alpha_{2i-2}$;

  \item An {\em above} component, which has one endpoint on each of the arcs
    $\beta_{i-1}$ and $\beta_i$, and intersects the arc $\alpha_{2i-3}$ but not
    the arc $\alpha_{2i-2}$; or

  \item A {\em below} component, which has one endpoint on each of the arcs
    $\beta_{i-1}$ and $\beta_i$, and intersects the arc $\alpha_{2i-2}$ but not
    the arc $\alpha_{2i-3}$.
\end{itemize}

\begin{remark}
\label{rmk:Delta_i} Clearly there cannot be both left loop and right loop
  components. It follows immediately from~(\ref{eq:dynn-coords}) that there
  are~$|b_{i-1}|$ loop components, which are left loops if $b_{i-1}<0$, and
  right loops if $b_{i-1}>0$; and hence that there are
  $\alpha_{2i-3}-|b_{i-1}|$ above components and $\alpha_{2i-2}-|b_{i-1}|$
  below components.
\end{remark}

The intersection numbers $(\alpha\,;\,\beta)$ (and hence the multicurve~$\cL$) can be recovered from the Dynnikov coordinates $(a;\,b) \in
\Z^{2n}$ using the formul\ae

  \begin{eqnarray}
    \label{eq:beta}
    \beta_i &=& -2\,\,\sum_{k=0}^{i-1} b_k \quad\text{ and} \\
    \label{eq:alpha}
    \alpha_i &=&
      \begin{cases} (-1)^ia_{\lceil i/2 \rceil} + \frac{\beta_{\lceil i/2
        \rceil}}{2} & \text{ if } b_{\lceil i/2 \rceil}\ge 0, \\
        (-1)^ia_{\lceil i/2 \rceil} + \frac{\beta_{1 + \lceil i/2 \rceil}}{2} &
        \text{ if } b_{\lceil i/2 \rceil}\le 0,
      \end{cases}
  \end{eqnarray} where $\lceil x \rceil$ denotes the smallest integer which is
not less than~$x$. \eqref{eq:beta} is immediate from~\eqref{eq:dynn-coords} and
the observation that $\beta_0=0$; while~\eqref{eq:alpha} follows
from~\eqref{eq:dynn-coords} and the equation $\alpha_{2i}+\alpha_{2i-1} =
\max(\beta_i, \beta_{i+1})$.

\begin{remark}
\label{rmk:reduced-coords} We have $a_0 = a_{n-1} = 0$ (since $\alpha_{-1} =
  \alpha_0$ and $\alpha_{2n-3}=\alpha_{2n-2})$; and \mbox{$\sum_{i=0}^{n-1} b_i
  = 0$} (since $\beta_0 = \beta_n = 0$). In fact there is one further relation
  \begin{equation}
  \label{eq:b_0} b_0 = - \max_{1\le k\le n-2} \left( |a_k| + b_k^+ +
    \sum_{j=1}^{k-1} b_j
    \right)
  \end{equation} (where $x^+ := \max(x,0)$), which arises from the fact that no
  component of a multicurve can enclose all~$n$ punctures (see for
  example Lemma~1 of~\cite{paper1}). It follows that $\cL$ can be described by
  its \emph{reduced} Dynnikov coordinates $(a_1,...,a_{n-2},
  b_1,...,b_{n-2})\in\Z^{2n-4}$: we can recover the (extended) coordinates by
  setting $a_0=a_{n-1}=0$, defining $b_0$ using~\eqref{eq:b_0}, and finally
  setting $b_{n-1} = -\sum_{j=0}^{n-2} b_j$. The reduced Dynnikov coordinate
  system gives a bijection $\rho_r\colon\cL_n\to\Z^{2n-4}$.
\end{remark}

\begin{remark}
\label{rmk:disjoint-union} If $\cL\ind{1},\ldots,\cL\ind{N}\in\cL_n$ have
  $\iota(\cL\ind{k},
  \cL\ind{\ell}) = 0$ for all~$k$ and $\ell$, then there are pairwise mutually
  disjoint representatives $L\ind{1},\ldots,L\ind{N}$ of the multicurves. We
  write $\cL = \bigsqcup_{k=1}^N \cL\ind{k}$ for the multicurve represented by
  the disjoint union $\bigsqcup_{k=1}^N L\ind{k}$, and observe that $\rho(\cL)
  =
  \sum_{k=1}^N \rho(\cL\ind{k})$.
\end{remark}

The following notation will be useful when we discuss the complexity of
algorithms involving Dynnikov coordinates.
\begin{notn}[$|\cL|$] Let~$\cL\in\cL_n$ with $\rho(\cL)=(a\,;\,b)$. We write
$|\cL| =
\sum_{i=0}^{n-1} (|a_i|+|b_i|)$.

\end{notn}

\subsection{The action of the braid group} The mapping class group $\MCG(D_n)$
of $D_n$ is isomorphic to the $n$-braid group~$B_n$ modulo its
center~\cite{emil2}, so that elements of $\MCG(D_n)$ can be represented in
terms of the Artin braid generators $\sigma_i$ ($1\le i\le n-1$). In this paper
we adopt the convention of Birman's book~\cite{birman}, that $\sigma_i$
exchanges punctures~$i$ and~$i+1$ in the counter-clockwise direction.

The action of $\MCG(D_n)$ on $\cL_n$ can be calculated using {\em update rules}
  (see for example~\cite{D02, M06, or08, paper1, paper2}), which describe how
  Dynnikov coordinates transform under the action of the Artin generators and
  their inverses. In this paper we only need the transformation under the
  positive generators~$\sigma_i$, which is given by Theorem~\ref{thm:update}
  below. In this theorem statement we again use the notation $x^+$ to denote
  $\max(x,0)$.

\begin{thm}[Update rules for positive generators]
\label{thm:update} Let $\cL\in\cL_n$ have Dynnikov coordinates $(a;\,b)$, and
let $1\le i\le n-1$. Denote by $(a';\,b')$ the Dynnikov coordinates of the
multicurve $\sigma_i(\cL)$. Then $a_j' = a_j$ and $b_j' = b_j$ for all
$j\not\in\{i-1,i\}$, and

\begin{equation}
  \label{eq:pos-update}
  \begin{aligned} a_{i-1}' &= \max(a_{i-1} + b_{i-1}^+,\,\, a_i + b_{i-1}),\\
    a_i' &= b_i - \max(-a_{i-1},\,\,b_i^+ - a_i),\\ b_{i-1}' &= a_i + b_{i-1} +
    b_i - \max(a_{i-1} + b_{i-1}^+ + b_i^+,\,\, a_i + b_{i-1}),\\ b_i' &=
    \max(a_{i-1} + b_{i-1}^+ + b_i^+,\,\, a_i + b_{i-1}) - a_i.
  \end{aligned}
\end{equation}

\end{thm}

\section{Geometric intersection number with an elementary multicurve}
\label{sec:int-num-elementary}
\begin{defn}[Elementary multicurve $\cL_{i,j}$] Let $1\le i < j \le n$, with
  $(i,j)\not=(1,n)$. The \emph{elementary multicurve $\cL_{i,j}\in\cL_n$ about
  punctures~$i$ through~$j$} is the multicurve with Dynnikov coordinates
  $(a\,;\,b)\in\Z^{2n}$ which are all zero except for $b_{i-1} = -1$ and
  $b_{j-1} = 1$. (This is equivalent to saying that $\cL_{i,j}$ is represented
  by a simple closed curve which bounds a disk containing punctures~$i$
  through~$j$, and intersects the diameter of~$D_n$ exactly twice.)
\end{defn}

In this section we obtain a formula for $\iota(\cL, \cL_{i,j})$,
given a multicurve $\cL\in\cL_n$. We start by introducing some
notation.

\begin{notn}[$A_i$, $B_i$, $A_{\ell, m}$, and $B_{\ell, m}$]
\label{notn:A_i-B_i} Let $\cL\in\cL_n$ be a multicurve with Dynnikov
  coordinates~$(a\,;\,b)$ and intersection numbers $(\alpha\,;\,\beta)$ with
  the Dynnikov arcs. For each $i$ with $1\le i\le n$, we write
  \[ A_i = \alpha_{2i-3} - |b_{i-1}| \qquad\text{and}\qquad B_i = \alpha_{2i-2}
    - |b_{i-1}|.
  \] For each $\ell$ and $m$ with $1\le \ell\le m \le n$, write
  \begin{equation}
  \label{eq:A_i,j-B_i,j} A_{\ell, m} = \min_{\ell\le k\le m} A_k \qquad
    \text{and} \qquad B_{\ell,m} = \min_{\ell\le k\le m} B_k,
  \end{equation}
\end{notn}

\begin{remark}
\label{rmk:Delta_l,m} By Remark~\ref{rmk:Delta_i}, $A_i$ and $B_i$ are,
  respectively, the number of above and below components in~$\Delta_i$.

  Given $1\le \ell\le m\le n$, let $\Delta_{\ell, m} = \bigcup_{i=\ell}^m
  \Delta_i$ be the subset of $D_n$ bounded by $\beta_{\ell-1}$ and $\beta_m$.
  If~$L$ is a minimal representative of a multicurve~$\cL$, then a
  component of $L\cap \Delta_{\ell, m}$ is called a \emph{large over}
  (respectively \emph{large under}) component if it lies entirely above
  (respectively below) the diameter of~$D_n$. Since large over components are
  the highest components in each of the~$\Delta_i$, it follows that $A_{\ell,
  m}$ is the number of large over components of~$L\cap\Delta_{\ell, m}$; and
  analogously $B_{\ell, m}$ is the number of large under components.
\end{remark}

\begin{lem}[Intersections with an elementary multicurve]
\label{lem:intersect-elementary} Let $1\le i < j \le n$ with $(i,j)\not=(1,n)$;
  and let $\cL\in\cL_n$ be a multicurve. Write
  \begin{align*} R &= \min(A_{i, j-1} - A_{i,j},\,\, B_{i, j-1} - B_{i, j},\,\,
  b_{j-1}^+), \,\, \text{ and}\\ L &= \min(A_{i+1, j} - A_{i, j},\,\, B_{i+1,
  j} - B_{i, j},\,\, (-b_{i-1})^+).
  \end{align*} where $A_{i,j}$ and $B_{i,j}$ are defined
  by~\eqref{eq:A_i,j-B_i,j}. Then
  \[
    \iota(\cL, \cL_{i, j}) = \beta_{i-1} + \beta_j - 2(R + L + A_{i,j} +
    B_{i,j}).
  \]
\end{lem}

\begin{proof} Let $C_{i,j}$ be a minimal representative of~$\cL_{i,j}$, and let
  $L$ be a representative of~$\cL$ which is minimal with respect both to the
  Dynnikov arcs and to $C_{i,j}$. Every component of $L\cap\Delta_{i,j}$ is
  therefore either disjoint from $C_{i,j}$ or intersects it exactly twice.

  Components of $L\cap\Delta_{i,j}$ which are disjoint from $C_{i,j}$ are
  precisely:
  \begin{itemize}
    \item Components which are contained in the interior of $\Delta_{i,j}$;
    \item Large over and large under components, which have one endpoint
    on~$\beta_{i-1}$ and one on~$\beta_j$;
    \item \emph{Large right loop} components, which have both endpoints on the
    arc $\beta_{i-1}$ and intersect the diameter of~$D_n$ only between
    $\beta_j$ and puncture~$j$; and
    \item \emph{Large left loop} components, which have both endpoints on the
    arc $\beta_j$ and intersect the diameter of~$D_n$ only between
    $\beta_{i-1}$ and puncture~$i$.
  \end{itemize}

  The total number of large over and under components is~$A_{i,j} + B_{i,j}$,
  by Remark~\ref{rmk:Delta_l,m}. The proof can therefore be completed by
  showing that the number of large right (respectively left) loop components
  is~$R$ (respectively~$L$).

  Since $A_{i, j-1} - A_{i,j}$ (respectively $B_{i,j-1} - B_{i,j}$) is the
  number of large over (respectively under) components of $L\cap\Delta_{i,
  j-1}$ which are not contained in large over (respectively under) components
  of $L\cap \Delta_{i,j}$; and $b_{j-1}^+$ is the number of right loop
  components of $\Delta_j$ (see Remark~\ref{rmk:Delta_i}), the number of large
  right loop components is the minimum of these three numbers, namely~$R$ (see
  Figure~\ref{fig:deltas}). The argument that there are~$L$ large left loop
  components is analogous.

\begin{figure}[htbp]
  \begin{center}
    \includegraphics[width=0.75\textwidth]{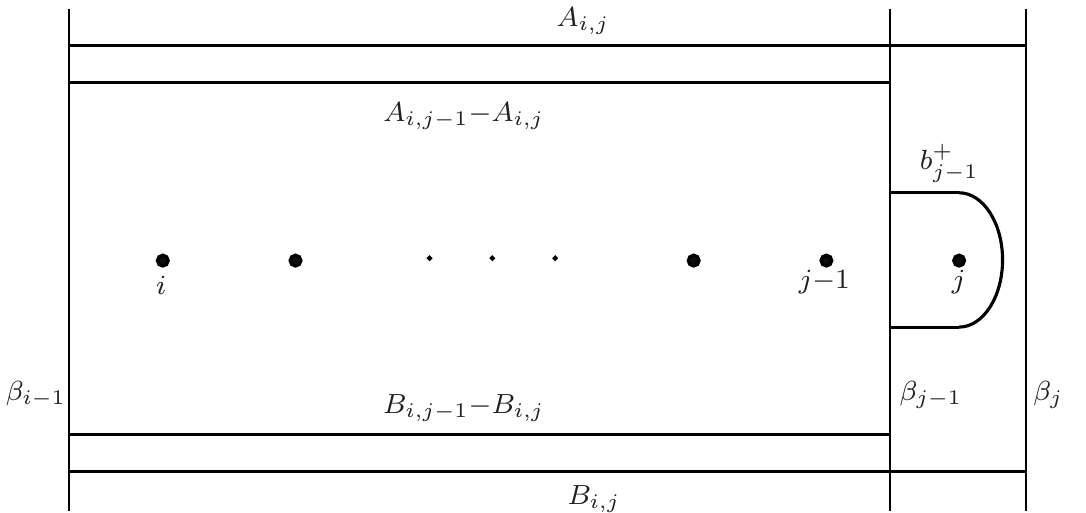}
    \caption{The number of large right loop components in $\Delta_{i,j}$}
    \label{fig:deltas}
  \end{center}
\end{figure}

\end{proof}

\section{Cumplido's relaxation algorithm}
\label{sec:cumplido}
\begin{defn}[Relaxed multicurve] A multicurve~$\cL\in\cL_n$, with
  Dynnikov coordinates $(a\,;\,b) \in \Z^{2n}$, is said to be \emph{relaxed} if
  $a_i=0$ for all~$i$.
\end{defn}

We observe that a multicurve is relaxed if and only if it is a
disjoint union $\cL =\bigsqcup_{k=1}^N \cL_{i_k, j_k}$ of elementary
multicurves. On the one hand, it is immediate that such a multicurve has
$a_i=0$ for all~$i$. Conversely, the following algorithm --- essentially a
bracket matching algorithm --- parses a relaxed multicurve as such a disjoint
union.

\begin{alg}[Parsing a relaxed multicurve]
  \label{alg:parse} Let~$\cL\in\cL_n$ be a relaxed multicurve with Dynnikov
  coordinates~$(a\,;\, b)\in\Z^{2n}$. The following algorithm returns a list
  \mbox{$\cC=( (i_i, j_1), \ldots, (i_N, j_N) )$} with the property that
  $\cL=\bigsqcup_{k=1}^N \cL_{i_k, j_k}$.

  \begin{algorithmic}[1]
    \myState{ $\cC\gets \mbox{ empty list}$}
    \myState{ $s\gets \mbox{ empty stack}$}
    \For{$i$ from $1$ to $n$}
      \If{$b_{i-1}<0$}
        \myState{ push $-b_{i-1}$ copies of $i$ onto $s$}
      \Else
        \myState{ pop $b_{i-1}$ top entries $\ell_1,\ldots,\ell_{b_{i-1}}$
        from~$s$}
        \myState{ add $(\ell_1, i),\ldots,(\ell_{b_{i-1}}, i)$ to $\cC$}
      \EndIf
    \EndFor
    \myState{ \textbf{return} $\cC$}
  \end{algorithmic}
\end{alg} Note that since $\sum_{k=0}^{i-1}b_k = -\beta_i/2 \le 0$ for each~$i$
by~\eqref{eq:beta}, the stack~$s$ is never empty at line~7. Moreover, if
$(i,j)$ and $(i',j')$ are in~$\cC$ with $i<i'$, then either $j<i'$ or $j\ge
j'$, so that $\iota(\cL_{i,j}, \cL_{i',j'}) = 0$. The
multicurves~$\cL_{i_k,j_k}$ can therefore be realised disjointly by
Remark~\ref{rmk:disjoint-union}. That $\cL=\bigsqcup_{k=1}^N \cL_{i_k, j_k}$
has Dynnikov coordinates~$(a\,;\,b)$ then follows from
Remark~\ref{rmk:Delta_i}, since if~$L$ is a minimal representative of $\cL$,
then $L\cap \Delta_i$ has $-b_{i-1}$ left loop components if $b_{i-1}\le 0$,
and $b_{i-1}$ right loop components if $b_{i-1} \ge 0$.

Cumplido~\cite{Cumplido} gives an algorithm which takes as input the Dynnikov
coordinates $(a\,;\,b)\in\Z^{2n}$ of~$\cL\in\cL_n$, and produces as output a
braid $\beta\in B_n^+$ (the positive braid monoid) and a relaxed multicurve
$\cL' = \beta(\cL)$ (in fact, $\beta$ is the unique prefix-minimal positive
braid which relaxes~$\cL$ in this way).

\begin{alg}[Cumplido's relaxation algorithm]
\label{alg:cumplido} Let $(a\,;\,b)\in\Z^{2n}$ be the Dynnikov coordinates of
$\cL\in\cL_n$. The following algorithm returns~$\beta\in B_n^+$ and the
Dynnikov coordinates of a relaxed multicurve $\cL'$ with $\cL' = \beta(\cL)$.

\begin{algorithmic}[1]
  \myState{ $\beta\gets\operatorname{id}$}
  \myState{ $j\gets 1$}
  \While{$j<n$}
    \If{$a_j > a_{j-1}$}
      \myState{ $(a\,;\,b) \gets \sigma_j(a\,;\,b)$}
      \Comment{Use~\eqref{eq:pos-update}}
      \myState{ $\beta\gets\beta\cdot\sigma_j$}
      \myState{ $j\gets 1$}
    \Else
      \myState{ $j\gets j+1$}
    \EndIf
  \EndWhile
  \myState{ \textbf{return} $(\beta, (a\,;\,b))$}
\end{algorithmic}

\end{alg}

The following result is contained in Corollaries~44 and~46 of~\cite{Cumplido}.

\begin{thm}[Cumplido]
\label{thm:cumplido-complexity} Let~$\cL\in\cL_n$ and write $m =|\cL|$. Then
  Algorithm~\ref{alg:cumplido} requires $O(n^2m)$ arithmetic operations; and
  the word length of the relaxing braid~$\beta$ which it returns is $O(n^2m)$.
\end{thm}

\section{The geometric intersection number algorithm}
\label{sec:the-alg} It is now straightforward to state the geometric intersection number
algorithm, which relies on the fact that if $\cL\ind{1},\cL\ind{2}\in\cL_n$ and
$\beta\in B_n$, then $\iota(\beta(\cL\ind{1}), \beta(\cL\ind{2})) =
\iota(\cL\ind{1},
\cL\ind{2})$.

\begin{alg}[geometric intersection number algorithm]
\label{alg:int-num} Let $(a\ind{1}\,;\,b\ind{1})$ and $(a\ind{2}\,;\,b\ind{2})$
  be the Dynnikov coordinates of $\cL\ind{1},\cL\ind{2}\in\cL_n$. The following
  algorithm returns $\iota(\cL\ind{1}, \cL\ind{2})$.

  \begin{algorithmic}[1]
    \myState{ If $(a\ind{1}\,;\,b\ind{1})$, $(a\ind{2}\,;\,b\ind{2})$ are reduced
    coordinates, extend them.} \Comment{Use Remark~\ref{rmk:reduced-coords}}
    \myState{ Find $\beta\in B_n^+$ such that $\cL\indp{1} = \beta(\cL\ind{1})$
    is relaxed.} \Comment{Use Algorithm~\ref{alg:cumplido}}
    \myState{ Parse $\cL\indp{1} = \bigsqcup_{k=1}^N \cL_{i_k, j_k}$.} \Comment{Use
    Algorithm~\ref{alg:parse}}
    \myState{ Calculate the Dynnikov coordinates of $\cL\indp{2} =
    \beta(\cL\ind{2})$.} \Comment{Use~\eqref{eq:pos-update}}
    \myState{ Determine the intersection numbers $(\alpha\,;\,\beta)$ of
    $\cL\indp{2}$.} \Comment{Use~\eqref{eq:beta} and~\eqref{eq:alpha}}
    \myState{ \textbf{Return} $\sum_{i=1}^k \iota(\cL\indp{2}, \cL_{i_k, j_k})$.}
    \Comment{Use Lemma~\ref{lem:intersect-elementary}}
  \end{algorithmic}
\end{alg}

\begin{remark}
  The same algorithm may be used to compute the measure of a multicurve~$\cL$ with respect to a measured foliation~$(\cF, \mu)$ on~$D_n$,
  described by its Dynnikov coordinates $\rho(\cF,\mu)\in\R^{2n}$ (or its
  reduced Dynnikov coordinates $\rho_r(\cF,\mu)\in\R^{2n-4}$).
\end{remark}

\begin{thm}[Complexity of the geometric intersection number algorithm] Let
$\cL\ind{1},\cL\ind{2}\in\cL_n$, and write $m=|\cL\ind{1}| + |\cL\ind{2}|$.
Then Algorithm~\ref{alg:int-num} has complexity $O(m^2n^4)$.
\end{thm}

\begin{proof} We first observe that the number of arithmetic operations
(addition, subtraction, comparing, taking the maximum, or taking the minimum of
two integers) required for each of the six steps of the the algorithm is
$O(n^2m)$.

\begin{enumerate}[1.]
  \item Extending reduced coordinates using Remark~\ref{rmk:reduced-coords}
  requires $O(n)$ arithmetic operations.

  \item Algorithm~\ref{alg:cumplido} requires $O(n^2m)$ arithmetic operations
  by Theorem~\ref{thm:cumplido-complexity}. We observe that $\cL\ind{1}$, and
  hence $\cL\indp{1}$, has $O(m)$ components (each component must form
  left/right loops around at least two punctures, and hence contributes at
  least~2 to $|\cL|$ by Remark~\ref{rmk:Delta_i}).

  \item Parsing $\cL\indp{1}$ into its $O(m)$ components using
  Algorithm~\ref{alg:parse} requires $O(m+n)$ arithmetic operations.

  \item Each application of a braid generator $\sigma_j$ to $\cL\ind{2}$
  using~\eqref{eq:pos-update} requires $O(1)$ arithmetic operations. Since
  $\beta$ has $O(n^2m)$ generators by Theorem~\ref{thm:cumplido-complexity},
  calculating the Dynnikov coordinates of $\cL\indp{2}$ requires $O(n^2m)$
  arithmetic operations.

  \item Calculating the intersection numbers~$(\alpha\,;\,\beta)$
  using~\eqref{eq:beta} and~\eqref{eq:alpha} requires $O(n)$ arithmetic
  operations.

  \item For each $k$, determining the geometric intersection number
  $\iota(\cL_{i_k, j_k},
  \cL\indp{2})$ using Lemma~\ref{lem:intersect-elementary} requires $O(n)$
  operations. Calculating the sum of~$O(m)$ such therefore requires $O(mn)$
  arithmetic operations.

\end{enumerate}

By~\eqref{eq:pos-update}, there is a constant factor~$K$ such that
$|\sigma_j(\cL)| \le K |\cL|$ for every multicurve $\cL$ and
every~$j$. Therefore the integers involved in each arithmetic operation are
$O(K^{mn^2}m)$. Since each arithmetic operation on such integers has complexity
$O(mn^2)$, the complexity of Algorithm~\ref{alg:int-num} is $O(m^2n^4)$ as
required.

\end{proof}

The algorithm has been implemented as part of the second author's program
\texttt{Dynn}, available at
\url{http://pcwww.liv.ac.uk/maths/tobyhall/software}. Experimentally, it
appears to scale better than $m^2n^4$ for small values of~$n$ (up to~100)
and~$m$ (up to $10^6$): in this range, the time taken goes more like
$m^{1/3}n^{3/2}$.

\bibliographystyle{amsplain}
\bibliography{intersection_refs}
\end{document}